\documentclass[preprint,12pt]{elsarticle}

\usepackage{graphicx}

\usepackage{amsmath}
\usepackage{amsfonts}
\usepackage{color}

\begin{document}

\newcommand{\hide}[1]{}
\newcommand{\tbox}[1]{\mbox{\tiny #1}}
\newcommand{\half}{\mbox{\small $\frac{1}{2}$}}
\newcommand{\sinc}{\mbox{sinc}}
\newcommand{\const}{\mbox{const}}
\newcommand{\trc}{\mbox{trace}}
\newcommand{\intt}{\int\!\!\!\!\int }
\newcommand{\ointt}{\int\!\!\!\!\int\!\!\!\!\!\circ\ }
\newcommand{\eexp}{\mbox{e}^}
\newcommand{\bra}{\left\langle}
\newcommand{\ket}{\right\rangle}
\newcommand{\EPS} {\mbox{\LARGE $\epsilon$}}
\newcommand{\ar}{\mathsf r}
\newcommand{\im}{\mbox{Im}}
\newcommand{\re}{\mbox{Re}}
\newcommand{\bmsf}[1]{\bm{\mathsf{#1}}}
\newcommand{\mpg}[2][1.0\hsize]{\begin{minipage}[b]{#1}{#2}\end{minipage}}

\newcommand{\CC}{\mathbb{C}}
\newcommand{\NN}{\mathbb{N}}
\newcommand{\PP}{\mathbb{P}}
\newcommand{\RR}{\mathbb{R}}
\newcommand{\QQ}{\mathbb{Q}}
\newcommand{\ZZ}{\mathbb{Z}}

\newcommand{\p}{\partial}

\renewcommand{\a}{\alpha}
\renewcommand{\b}{\beta}
\renewcommand{\d}{\delta}
\newcommand{\D}{\Delta}
\newcommand{\g}{\gamma}
\newcommand{\G}{\Gamma}
\renewcommand{\th}{\theta}
\renewcommand{\l}{\lambda}
\renewcommand{\L}{\Lambda}
\renewcommand{\O}{\Omega}
\newcommand{\s}{\sigma}

\newtheorem{theorem}{Theorem}
\newtheorem{acknowledgement}[theorem]{Acknowledgement}
\newtheorem{algorithm}[theorem]{Algorithm}
\newtheorem{axiom}[theorem]{Axiom}
\newtheorem{claim}[theorem]{Claim}
\newtheorem{conclusion}[theorem]{Conclusion}
\newtheorem{condition}[theorem]{Condition}
\newtheorem{conjecture}[theorem]{Conjecture}
\newtheorem{corollary}[theorem]{Corollary}
\newtheorem{criterion}[theorem]{Criterion}
\newtheorem{definition}[theorem]{Definition}
\newtheorem{example}[theorem]{Example}
\newtheorem{exercise}[theorem]{Exercise}
\newtheorem{lemma}[theorem]{Lemma}
\newtheorem{notation}[theorem]{Notation}
\newtheorem{problem}[theorem]{Problem}
\newtheorem{proposition}[theorem]{Proposition}
\newtheorem{remark}[theorem]{Remark}
\newtheorem{solution}[theorem]{Solution}
\newtheorem{summary}[theorem]{Summary}
\newenvironment{proof}[1][Proof]{\noindent\textbf{#1.} }{\ \rule{0.5em}{0.5em}}


\journal{DAM}

\begin{frontmatter}



\title{Optimal inequalities and extremal problems on the general  Sombor index}


\author[UAGRO]{Juan C. Hern\'andez}\ead{jcarloshg@gmail.com}
\author[UIII]{Jos\'e M. Rodr\'iguez}\ead{jomaro@math.uc3m.es}
\author[UAGRO]{Omar Rosario}\ead{omarrosarioc@gmail.com}
\author[UAGRO]{Jos\'e M. Sigarreta}\ead{josemariasigarretaalmira@hotmail.com}

\address[UAGRO]{Facultad de Matem\'aticas, Universidad Aut\'onoma de Guerrero,
Carlos E. Adame No.54 Col. Garita, 39650 Acalpulco Gro., Mexico}
\address[UIII]{Departamento de Matem\'aticas, Universidad Carlos III de Madrid,
 Avenida de la Universidad 30, 28911 Legan\'es, Madrid, Spain}

\begin{abstract}
In this work we obtain new lower and upper optimal bounds of general Sombor indices.
Specifically, we have inequalities for these indices
relating them with other indices: the first Zagreb index,
the forgotten index and the first variable Zagreb index.
Finally, we solve some extremal problems for general Sombor indices.
\end{abstract}


\end{frontmatter}

\section{Introduction}

Topological indices have become an important research topic associated with the study of their mathematical and computational properties and, fundamentally, for their multiple applications to various areas of knowledge (see, e.g., \cite{Tod-09,Gut-14,Gut-18}). Within the study of mathematical properties, we will contribute to the study of inequalities and optimization problems associated with topological indices.
Our main goal are the Sombor indices, introduced by Gutman in \cite{Gut-21}.

In what follows, $G=\left( V\left( G\right) ,E\left( G\right) \right) $
will be a finite undirected graph, and we will assume that each vertex has at least a neighbor.
We denote by $d_{w}$ the degree of the vertex $w$, i.e., the number of neighbors of $w$.
We denote by $uv$ the edge joining the vertices $u$ and $v$ (or $v$ and $u$).
For each graph $G$, its
\emph{Sombor index} is
\[
SO\left( G\right) =\sum\limits_{uv\in E\left( G\right) }\sqrt{d_{u}^{2}+ d_{v}^{2}} \, .
\]
In the same paper is also defined the \emph{reduced Sombor index} by
\[
SO_{red}\left( G\right) =\sum\limits_{uv\in E\left( G\right) }\sqrt{\left(
d_{u}-1\right) ^{2}+\left( d_{v}-1\right) ^{2}}\,.
\]
In \cite{Red-21} it is shown that these indices have a good predictive potential.

Also, the \emph{modified Sombor index} of $G$ was proposed in \cite{KG21} as
\begin{equation}
^mSO(G)
= \sum_{uv\in E(G)} \frac{1}{\sqrt{d_u^2 + d_v^2}} .
\label{mSO}
\end{equation}
In addition, two other Sombor indices have been introduced:
the \emph{first Banhatti-Sombor index} \cite{LZKM21}
\begin{equation}
BSO(G)
= \sum_{uv\in E(G)} \sqrt{\frac{1}{d_u^2} + \frac{1}{d_v^2} }
\label{BSO}
\end{equation}
and the $\alpha$-\emph{Sombor index} \cite{RDA21}
\begin{equation}
SO_\alpha(G)
= \sum_{uv\in E(G)} (d_u^\alpha + d_v^\alpha)^{1/\alpha} ,
\label{pSO}
\end{equation}
here $\alpha \in \RR \setminus \{0\}$.
In fact, there is a general index that includes most Sombor indices listed above: the first $(\alpha,\beta)\,$--$\,KA$ \emph{index} of $G$
which was introduced in \cite{K19} as
\begin{equation}
KA_{\alpha,\beta}(G)
= KA^1_{\alpha,\beta}(G)
= \sum_{uv\in E(G)} \left( d_u^\alpha + d_v^\alpha \right)^\beta ,
\label{KA1}
\end{equation}
with $\alpha,\beta \in \RR$.
Note that $SO(G)=KA_{2,1/2}(G)$, $^mSO(G)=KA_{2,-1/2}(G)$, $BSO(G)=KA_{-2,1/2}(G)$, and
$SO_\alpha(G)=KA_{\alpha,1/\alpha}(G)$. Also, we note that $KA_{1,\beta}(G)$ equals the general
sum-connectivity index \cite{ZT10}
$
\chi_\beta(G) = \sum_{uv \in E(G)} (d_u + d_v)^\beta .
$
Reduced versions of $SO(G)$, $^mSO(G)$ and $KA_{\alpha,\beta}(G)$ were also introduced
in \cite{Gut-21,KG21,K21}, e.g., the reduced $(\alpha,\beta)\,$--$\,KA$ index is
$$
_{red}KA_{\alpha,\beta}(G)
= \sum_{uv\in E(G)} \big( (d_u-1)^\alpha + (d_v-1)^\alpha \big)^\beta .
$$
If $\a<0$, then $_{red}KA_{\alpha,\beta}(G)$ is just defined for graphs without pendant vertices
(recall that a vertex is said pendant if its degree is $1$).

In \cite{Gut-21}, Gutman initiates the study of the mathematical properties of $SO$.
Many papers have continued this study, see e.g.,
\cite{Cru-21}, \cite{Cru-21b}, \cite{Das-21}, \cite{Gut-21b}, \cite{Mil-21}, \cite{Rada-21}, \cite{RDA-21}.

Our main aim is to obtain new bounds of Sombor indices, and characterize the graphs where equality occurs.
In particular, we have obtained bounds for Somobor indices
relating them with the first Zagreb index,
the forgotten index and the first variable Zagreb index.
Also, we solve some extremal problems for Sombor indices.

\section{Inequalities for the Sombor indices}

The following inequalities are known for $x,y > 0$:
$$
\begin{aligned}
x^a + y^a
< (x + y)^a
\le 2^{a-1}(x^a + y^a )
\qquad & \text{if } \, a > 1,
\\
2^{a-1}(x^a + y^a )
\le (x + y)^a
< x^a + y^a
\qquad & \text{if } \, 0<a < 1,
\\
(x + y)^a
\le 2^{a-1}(x^a + y^a )
\qquad & \text{if } \, a < 0,
\end{aligned}
$$
and the equality in the second, third or fifth bound is tight for each $a$ if and only if $x=y$.

These inequalities allow to obtain the following result relating $KA$ indices.

\begin{theorem} \label{t:ineq1}
Let $G$ be any graph and $\a,\b,\l \in \RR \setminus \{0\}$.
Then
$$
\begin{aligned}
KA_{\a\b/\l,\,\l}(G)
< KA_{\a,\b}(G)
\le 2^{\b-\l}KA_{\a\b/\l,\,\l}(G)
\qquad & \text{if } \, \b>\l,\, \b\l > 0,
\\
2^{\b-\l}KA_{\a\b/\l,\,\l}(G)
\le KA_{\a,\b}(G)
< KA_{\a\b/\l,\,\l}(G)
\qquad & \text{if } \, \b<\l,\, \b\l > 0,
\\
KA_{\a,\b}(G)
\le 2^{\b-\l}KA_{\a\b/\l,\,\l}(G)
\qquad & \text{if } \, \b < 0, \l > 0,
\\
KA_{\a,\b}(G)
\ge 2^{\b-\l}KA_{\a\b/\l,\,\l}(G)
\qquad & \text{if } \, \b > 0, \l < 0,
\end{aligned}
$$
and the equality in the second, third, fifth or sixth bound is tight for each $\a,\b,\l$
if and only if $G$ has regular connected components.
\end{theorem}

\begin{proof}
If $a=\b/\l$, $x=d_u^{\a}$ and $y=d_v^{\a}$, then the previous inequalities give
$$
\begin{aligned}
d_u^{\a\b/\l} + d_v^{\a\b/\l}
< (d_u^{\a} + d_v^{\a})^{\b/\l}
\le 2^{\b/\l-1}(d_u^{\a\b/\l} + d_v^{\a\b/\l} )
\qquad & \text{if } \, \b/\l > 1,
\\
2^{\b/\l-1}(d_u^{\a\b/\l} + d_v^{\a\b/\l} )
\le (d_u^{\a} + d_v^{\a})^{\b/\l}
< d_u^{\a\b/\l} + d_v^{\a\b/\l}
\qquad & \text{if } \, 0 < \b/\l <1,
\\
(d_u^{\a} + d_v^{\a})^{\b/\l}
\le 2^{\b/\l-1}(d_u^{\a\b/\l} + d_v^{\a\b/\l} )
\qquad & \text{if } \, \b/\l < 0,
\end{aligned}
$$
and the equality in the second, third or fifth bound is tight if and only if $d_u=d_v$.

Hence, we obtain
$$
\begin{aligned}
(d_u^{\a\b/\l} + d_v^{\a\b/\l})^{\l}
< (d_u^{\a} + d_v^{\a})^{\b}
\le 2^{\b-\l}(d_u^{\a\b/\l} + d_v^{\a\b/\l} )^{\l}
\qquad & \text{if } \, \b/\l > 1, \,\l>0,
\\
2^{\b-\l}(d_u^{\a\b/\l} + d_v^{\a\b/\l} )^{\l}
\le (d_u^{\a} + d_v^{\a})^{\b}
< (d_u^{\a\b/\l} + d_v^{\a\b/\l})^{\l}
\qquad & \text{if } \, \b/\l > 1, \,\l<0,
\\
2^{\b-\l}(d_u^{\a\b/\l} + d_v^{\a\b/\l} )^{\l}
\le (d_u^{\a} + d_v^{\a})^{\b}
< (d_u^{\a\b/\l} + d_v^{\a\b/\l})^{\l}
\qquad & \text{if } \, 0 < \b/\l <1, \, \l>0,
\\
(d_u^{\a\b/\l} + d_v^{\a\b/\l})^{\l}
< (d_u^{\a} + d_v^{\a})^{\b}
\le 2^{\b-\l}(d_u^{\a\b/\l} + d_v^{\a\b/\l} )^{\l}
\qquad & \text{if } \, 0 < \b/\l <1, \, \l<0,
\\
(d_u^{\a} + d_v^{\a})^{\b}
\le 2^{\b-\l}(d_u^{\a\b/\l} + d_v^{\a\b/\l} )^{\l}
\qquad & \text{if } \, \b < 0, \l > 0,
\\
(d_u^{\a} + d_v^{\a})^{\b}
\ge 2^{\b-\l}(d_u^{\a\b/\l} + d_v^{\a\b/\l} )^{\l}
\qquad & \text{if } \, \b > 0, \l < 0,
\end{aligned}
$$
and the equality in the non-strict inequalities is tight if and only if $d_u=d_v$.

If we sum on $uv \in E(G)$ these inequalities, then we obtain $(1)$.
\end{proof}

\begin{remark}
Note that the excluded case $\b=\l$ in Theorem \ref{t:ineq1} is not interesting, since
$KA_{\a\b/\l,\,\l}(G) = KA_{\a,\b}(G)$ if $\b=\l$.
\end{remark}

The argument in the proof of Theorem \ref{t:ineq1} also allows to obtain the following result relating reduced $KA$ indices.

\begin{theorem} \label{t:ineq2}
Let $G$ be any graph and $\a,\b,\l \in \RR \setminus \{0\}$.
If $\a<0$ or $\,\a\b\l < 0$, we also assume that $G$ does not have pendant vertices.
Then
$$
\begin{aligned}
{}_{red}KA_{\a\b/\l,\,\l}(G)
< {}_{red}KA_{\a,\b}(G)
\le 2^{\b-\l} {}_{red}KA_{\a\b/\l,\,\l}(G)
\qquad & \text{if } \, \b>\l, \, \b\l>0,
\\
2^{\b-\l} {}_{red}KA_{\a\b/\l,\,\l}(G)
\le {}_{red}KA_{\a,\b}(G)
< {}_{red}KA_{\a\b/\l,\,\l}(G)
\qquad & \text{if } \, \b<\l, \, \b\l>0,
\\
{}_{red}KA_{\a,\b}(G)
\le 2^{\b-\l} {}_{red}KA_{\a\b/\l,\,\l}(G)
\qquad & \text{if } \, \b < 0, \l > 0,
\\
{}_{red}KA_{\a,\b}(G)
\ge 2^{\b-\l} {}_{red}KA_{\a\b/\l,\,\l}(G)
\qquad & \text{if } \, \b > 0, \l < 0,
\end{aligned}
$$
and the equality in the second, third, fifth or sixth bound is tight for each $\a,\b,\l$
if and only if $G$ has regular connected components.
\end{theorem}

If we take $\b=1/\a$ and $\mu=1/\l$ in Theorem \ref{t:ineq1}, we obtain the following inequalities for
the $\a$-Sombor index.

\begin{corollary} \label{c:ineq1}
Let $G$ be any graph and $\a,\mu \in \RR \setminus \{0\}$.
Then
$$
\begin{aligned}
SO_{\mu}(G)
< SO_{\a}(G)
\le 2^{1/\a-1/\mu} SO_{\mu}(G)
\qquad & \text{if } \, \mu>\a, \, \a\mu>0,
\\
2^{1/\a-1/\mu} SO_{\mu}(G)
\le SO_{\a}(G)
< SO_{\mu}(G)
\qquad & \text{if } \, \mu<\a, \, \a\mu>0,
\\
SO_{\a}(G)
\le 2^{1/\a-1/\mu} SO_{\mu}(G)
\qquad & \text{if } \, \a < 0, \mu > 0,
\end{aligned}
$$
and the equality in the second, third or fifth bound is tight for each $\a,\mu$
if and only if $G$ has regular connected components.
\end{corollary}

Recall that one of the most studied topological indices is the \emph{first Zagreb index}, defined by
$$
M_1(G) = \sum_{u\in V(G)} d_u^2 .
$$

If we take $\mu=1$ in Corollary \ref{c:ineq1}, we obtain the following result.
\begin{corollary} \label{c:ineq1cuatris}
Let $G$ be any graph and $\a \in \RR \setminus \{0\}$.
Then
$$
\begin{aligned}
M_1(G)
< SO_{\a}(G)
\le 2^{1/\a-1} M_1(G)
\qquad & \text{if } \, 0 < \a < 1,
\\
2^{1/\a-1} M_1(G)
\le SO_{\a}(G)
< M_1(G)
\qquad & \text{if } \, \a >1,
\\
SO_{\a}(G)
\le 2^{1/\a-1} M_1(G)
\qquad & \text{if } \, \a < 0,
\end{aligned}
$$
and the equality in the second, third or fifth bound is tight for each $\a$
if and only if $G$ has regular connected components.
\end{corollary}

If we take $\a=2$, $\b=-1/2$ and $\l=1/2$ in Theorem \ref{t:ineq1}, we obtain the following inequality
relating the modified Sombor and the first Banhatti-Sombor indices.

\begin{corollary} \label{c:ineq1bis}
Let $G$ be any graph.
Then
$$
^m SO(G)
\le \frac12\, BSO(G)
$$
and the bound is tight if and only if $G$ has regular connected components.
\end{corollary}

In \cite{LZheng,LZhao,MN}, the \emph{first variable Zagreb index} is defined by
$$
M_1^{\a}(G) = \sum_{u\in V(G)} d_u^{\a} ,
$$
with $\a \in \mathbb{R}$.

Note that $M_1^{\a}$ generalizes numerous degree--based topological
indices which earlier have independently been studied.
For $\a=2$, $\a=3$, $\a=-1/2$, and $\a=-1$, \ $M_1^{\a}$ is, respectively, the ordinary first Zagreb index, the forgotten index $F$,
the zeroth--order Randi\'c index, and the inverse index $ID$ \cite{Gut-14,AGMM}.

The next result relates the $KA_{\a,\b}$ and $M_1^{\a+1}$ indices.

\begin{theorem} \label{t:gm1}
Let $G$ be any graph with maximum degree $\D$, minimum degree $\d$ and $m$ edges,
and $\a \in \RR \setminus \{0\}$, $\b > 0$.
Then
$$
\begin{aligned}
KA_{\a,\b}(G)
\ge \Big(
\frac{ M_1^{\a+1}(G) + 2 \D^{\a/2} \d^{\a/2} m}{ \sqrt{2}\, (\D^{\a/2} + \d^{\a/2})} \Big)^{2\b}
\qquad & \text{if } \, 0< \b < 1/2 ,
\\
KA_{\a,\b}(G)
\ge \Big(
\frac{ M_1^{\a+1}(G) + 2 \D^{\a/2} \d^{\a/2} m}{ \sqrt{2}\, (\D^{\a/2} + \d^{\a/2})} \Big)^{2\b} m^{1-2\b}
\qquad & \text{if } \, \b \ge 1/2 ,
\end{aligned}
$$
and the equality in the second bound is tight for some $\a,\b$
if and only if $G$ is a regular graph.
\end{theorem}

\begin{proof}
If $uv \in E(G)$ and $\a>0$, then
$$
\sqrt{2}\, \d^{\a/2}
\le \sqrt{d_u^{\a} + d_v^{\a}}
\le \sqrt{2}\, \D^{\a/2} .
$$
If $\a<0$, then the converse inequalities hold.
Hence,
$$
\begin{aligned}
\Big( \sqrt{d_u^{\a} + d_v^{\a}} - \sqrt{2}\, \d^{\a/2} \Big)
\Big( \sqrt{2}\, \D^{\a/2} - \sqrt{d_u^{\a} + d_v^{\a}} \, \Big)
\ge 0 ,
\\
\sqrt{2}\, (\D^{\a/2} + \d^{\a/2}) \sqrt{d_u^{\a} + d_v^{\a}}
\ge d_u^{\a} + d_v^{\a} + 2 \D^{\a/2} \d^{\a/2}.
\end{aligned}
$$
Since
$$
\sum_{uv \in E(G)} \big( d_u^{\a} + d_v^{\a} \big)
= \sum_{u \in V(G)} d_u\, d_u^{\a}
= \sum_{u \in V(G)} d_u^{\a+1}
= M_1^{\a+1}(G),
$$
If $0< \b < 1/2$, then $1/(2\b) > 1$ and
$$
\begin{aligned}
\sum_{uv \in E(G)} \sqrt{d_u^{\a} + d_v^{\a}}
& = \sum_{uv \in E(G)} \big( (d_u^{\a} + d_v^{\a})^\b \big)^{1/(2\b)}
\\
& \le \Big( \sum_{uv \in E(G)} (d_u^{\a} + d_v^{\a})^\b \Big)^{1/(2\b)}
= KA_{\a,\b}(G)^{1/(2\b)}.
\end{aligned}
$$
Consequently, we obtain
$$
KA_{\a,\b}(G)^{1/(2\b)}
\ge \frac{ M_1^{\a+1}(G) + 2 \D^{\a/2} \d^{\a/2} m}{ \sqrt{2}\, (\D^{\a/2} + \d^{\a/2})} \,.
$$
If $\b \ge 1/2$, then $2\b \ge 1$ and H\"older inequality gives
$$
\begin{aligned}
\sum_{uv \in E(G)} \sqrt{d_u^{\a} + d_v^{\a}}
& = \sum_{uv \in E(G)} \big( (d_u^{\a} + d_v^{\a})^\b \big)^{1/(2\b)}
\\
& \le \Big( \sum_{uv \in E(G)} (d_u^{\a} + d_v^{\a})^\b \Big)^{1/(2\b)}
\Big( \sum_{uv \in E(G)} 1^{2\b/(2\b-1)} \Big)^{(2\b-1)/(2\b)}
\\
& = m^{(2\b-1)/(2\b)}  KA_{\a,\b}(G)^{1/(2\b)}.
\end{aligned}
$$
Consequently, we obtain
$$
KA_{\a,\b}(G)^{1/(2\b)}
\ge \frac{ M_1^{\a+1}(G) + 2 \D^{\a/2} \d^{\a/2} m}{ \sqrt{2}\, (\D^{\a/2} + \d^{\a/2})} \, m^{(1-2\b)/(2\b)}.
$$
If $G$ is regular, then
$$
\begin{aligned}
\Big( \frac{ M_1^{\a+1}(G) + 2 \D^{\a/2} \d^{\a/2} m}{ \sqrt{2}\, (\D^{\a/2} + \d^{\a/2})} \Big)^{2\b} m^{1-2\b}
& = \Big( \frac{ 2\D^{\a}m + 2 \D^{\a} m}{ \sqrt{2}\, 2\D^{\a/2}} \Big)^{2\b} m^{1-2\b}
\\
& = \Big( \sqrt{2}\, \D^{\a/2} m \Big)^{2\b} m^{1-2\b}
\\
& = \big( 2 \D^{\a} \big)^{\b} m
= KA_{\a,\b}(G) .
\end{aligned}
$$

If the equality in the second bound is tight for some $\a,\b,$
then we have $d_u^\a+d_v^\a=2\d^\a$ or $d_u^\a+d_v^\a=2\D^\a$ for each $uv \in E(G)$.
Also, the equality in H\"older inequality gives that there exists a constant $c$ such that
$d_u^\a+d_v^\a=c$ for every $uv \in E(G)$.
Hence, we have either
$d_u^\a+d_v^\a=2\d^\a$ for each edge $uv$ or
$d_u^\a+d_v^\a=2\D^\a$ for each edge $uv$,
and hence, $G$ is regular.
\end{proof}

If we take $\a=2$ and $\b=1/2$ in Theorem \ref{t:gm1} we obtain:

\begin{corollary} \label{c:gm1}
Let $G$ be any graph with maximum degree $\D$, minimum degree $\d$ and $m$ edges.
Then
$$
\begin{aligned}
SO(G)
\ge \frac{ F(G) + 2 \D \d m}{ \sqrt{2}\, (\D + \d)} \,,
\end{aligned}
$$
and the bound is tight if and only if $G$ is regular.
\end{corollary}

In order to prove Theorem \ref{t:holder} below we need an additional technical result.
In \cite[Theorem 3]{BMRS} appears a converse of H\"older inequality, which in the discrete case can be stated as follows \cite[Corollay 2]{BMRS}.

\begin{proposition} \label{c:holder}
Consider constants $0<\a \le \b$ and $1<p,q<\infty$ with $1/p+1/q=1$.
If $w_k,z_k\ge 0$ satisfy $\a z_k^q \le w_k^p \le \b z_k^q$ for $1\le k \le n$, then
	$$
	\Big(\sum_{k=1}^n w_k^p \Big)^{1/p} \Big(\sum_{k=1}^n z_k^q \Big)^{1/q}
	\le C_p(\a ,\b  ) \sum_{k=1}^n w_kz_k ,
	$$
where
	$$
	C_p(\a ,\b  )
	= \begin{cases}
		\,\displaystyle\frac{1}{p} \Big( \frac{\a }{\b  } \Big)^{1/(2q)} + \frac{1}{q} \Big( \frac{\b  }{\a } \Big)^{1/(2p)}\,,
& \quad \text{when } 1<p<2 ,
		\\
		\, & \,
		\\
		\,\displaystyle\frac{1}{p} \Big( \frac{\b  }{\a } \Big)^{1/(2q)} + \frac{1}{q} \Big( \frac{\a }{\b  } \Big)^{1/(2p)}\,,
& \quad \text{when } p \ge 2 .
	\end{cases}
	$$
	If $(w_1,\dots,w_n) \neq 0$, then the bound is tight if and only if $w_k^p=\a z_k^q$
for each $1\le k \le n$ and $\a =\b$.
\end{proposition}

The next result relates several $KA$ indices.

\begin{theorem} \label{t:holder}
Let $G$ be any graph, $\a,\b,\mu \in \RR$ and $p>1$.
Then
$$
\begin{aligned}
D_p^p \,KA_{\a,p(\b-\mu)}(G) & \, KA_{\a,p \mu /(p-1)}(G)^{p-1}
\\
& \le KA_{\a,\b}(G)^p
\le KA_{\a,p(\b-\mu)}(G) \, KA_{\a,p \mu /(p-1)}(G)^{p-1} ,
\end{aligned}
$$
where
$$
	D_p
	= \begin{cases}
		\,C_p\big( ( 2\d^{\a} )^{p(\b-\mu \frac{p }{p-1})}, ( 2\D^{\a} )^{p(\b-\mu \frac{p }{p-1})} \big)^{-1},
& \;\;\; \text{if } \a (\b-\mu \frac{p }{p-1}\,) \ge 0 ,
		\\
		\, & \,
		\\
		\,C_p\big( ( 2\D^{\a} )^{p(\b-\mu \frac{p }{p-1})}, ( 2\d^{\a} )^{p(\b-\mu \frac{p }{p-1})} \big)^{-1},
& \;\;\; \text{if } \a (\b-\mu \frac{p }{p-1}\,) < 0 ,
	\end{cases}
	$$
and $C_p$ is the constant in Proposition \ref{c:holder}.

The equality in the upper bound is tight for each $\a,\b,\mu,p$
if $G$ is a biregular graph.

The equality in the lower bound is tight for each $\a,\b,\mu,p$ with $\a (\b-\mu \frac{p }{p-1}\,) \neq 0$
if and only if $G$ is a regular graph.
\end{theorem}

\begin{proof}
H\"older inequality gives
$$
\begin{aligned}
KA_{\a,\b}(G)
& = \sum_{uv \in E(G)} \big( d_u^{\a} + d_v^{\a} \big)^{\b-\mu} \big( d_u^{\a} + d_v^{\a} \big)^{\mu}
\\
& \le \Big( \sum_{uv \in E(G)} \big( d_u^{\a} + d_v^{\a} \big)^{p(\b-\mu)} \Big)^{1/p}
\Big( \sum_{uv \in E(G)} \big( d_u^{\a} + d_v^{\a} \big)^{p \mu /(p-1)} \Big)^{(p-1)/p} ,
\\
KA_{\a,\b}(G)^p
& \le KA_{\a,p(\b-\mu)}(G) \, KA_{\a,p \mu /(p-1)}(G)^{p-1} .
\end{aligned}
$$
If $G$ is biregular, we obtain
$$
\begin{aligned}
& KA_{\a,p(\b-\mu)}(G) \, KA_{\a,p \mu /(p-1)}(G)^{p-1}
= (\D^\a + \d^\a )^{p(\b-\mu)} m
\big( (\D^\a + \d^\a )^{p \mu /(p-1)} m \big)^{p-1}
\\
& = (\D^\a + \d^\a )^{p(\b-\mu)} (\D^\a + \d^\a )^{p \mu} m^p
= \big( (\D^\a + \d^\a )^{\b} m \big)^{p}
= KA_{\a,\b}(G)^p.
\end{aligned}
$$
Since
$$
\frac{\big( d_u^{\a} + d_v^{\a} \big)^{p(\b-\mu)}}{\big( d_u^{\a} + d_v^{\a} \big)^{p \mu /(p-1)}}
= \big( d_u^{\a} + d_v^{\a} \big)^{p(\b-\mu \frac{p }{p-1})} ,
$$
if $\a p(\b-\mu \frac{p }{p-1}\,) \ge 0$, then
$$
\big( 2\d^{\a} \big)^{p(\b-\mu \frac{p }{p-1})}
\le \frac{\big( d_u^{\a} + d_v^{\a} \big)^{p(\b-\mu)}}{\big( d_u^{\a} + d_v^{\a} \big)^{p \mu /(p-1)}}
\le \big( 2\D^{\a} \big)^{p(\b-\mu \frac{p }{p-1})} ,
$$
and if $\a p(\b-\mu \frac{p }{p-1}\,) < 0$, then
$$
\big( 2\D^{\a} \big)^{p(\b-\mu \frac{p }{p-1})}
\le \frac{\big( d_u^{\a} + d_v^{\a} \big)^{p(\b-\mu)}}{\big( d_u^{\a} + d_v^{\a} \big)^{p \mu /(p-1)}}
\le \big( 2\d^{\a} \big)^{p(\b-\mu \frac{p }{p-1})} .
$$
Proposition \ref{c:holder} gives
$$
\begin{aligned}
KA_{\a,\b}(G)
& = \sum_{uv \in E(G)} \big( d_u^{\a} + d_v^{\a} \big)^{\b-\mu} \big( d_u^{\a} + d_v^{\a} \big)^{\mu}
\\
& \ge D_p \Big( \sum_{uv \in E(G)} \big( d_u^{\a} + d_v^{\a} \big)^{p(\b-\mu)} \Big)^{1/p}
\Big( \sum_{uv \in E(G)} \big( d_u^{\a} + d_v^{\a} \big)^{p \mu /(p-1)} \Big)^{(p-1)/p} ,
\\
KA_{\a,\b}(G)^p
& \ge D_p^p \, KA_{\a,p(\b-\mu)}(G) \, KA_{\a,p \mu /(p-1)}(G)^{p-1} .
\end{aligned}
$$

Proposition \ref{c:holder} gives that the equality is tight in this last bound for some $\a,\b,\mu,p$ with $\a (\b-\mu \frac{p }{p-1}\,) \neq 0$ if and only if
$$
( 2\d^{\a} )^{p(\b-\mu \frac{p }{p-1})}
= ( 2\D^{\a} )^{p(\b-\mu \frac{p }{p-1})}
\quad
\Leftrightarrow
\quad
\d=\D,
$$
i.e., $G$ is regular.
\end{proof}

If we take $\b=0$ in Theorem \ref{t:holder}
we obtain the following result.

\begin{corollary} \label{c:holder2}
Let $G$ be any graph, $\a,\mu \in \RR$ and $p>1$.
Then
$$
\begin{aligned}
KA_{\a,-p\mu}(G) \, KA_{\a,p \mu /(p-1)}(G)^{p-1}
\ge m^p.
\end{aligned}
$$
The equality in the bound is tight for each $\a,\mu,p$
if $G$ is a biregular graph.
\end{corollary}

If we take $\a=2$, $\b=0$, $p=2$ and $\mu=1/4$ in Theorem \ref{t:holder}
we obtain the following result.

\begin{corollary} \label{c:holder3}
If $G$ is any graph, then
$$
\begin{aligned}
m^2
\le {}^m SO(G) \, SO(G)
\le \frac{(\D+\d)^2}{4\D\d} \, m^2.
\end{aligned}
$$
The equality in the upper bound is tight if and only if $G$ is regular.
The equality in the lower bound is tight if $G$ is a biregular graph.
\end{corollary}

Note that the following result improves the upper bound in Corollary \ref{c:ineq1cuatris} when $\a>1$.

\begin{theorem} \label{t:alb}
Let $G$ be any graph with maximum degree $\D$ and minimum degree $\d$, and $\a \ge 1$.
Then
$$
\begin{aligned}
2^{1/\a-1} M_1(G)
\le SO_\a(G)
& \le M_1(G) - ( 2 - 2^{1/\a} ) \d ,
\end{aligned}
$$
and the equality holds for some $\a>1$ in each bound if and only if $G$ is regular.
\end{theorem}
\begin{proof}
The lower bound follows from Corollary \ref{c:ineq1cuatris}.
Let us prove the upper bound.

First of all, we are going to prove that
\begin{equation} \label{eq:alb}
( x^\a + y^\a )^{1/\a} \le x + ( 2^{1/\a} - 1 )y
\end{equation}
for every $\a \ge 1$ and $x \ge y \ge 0$.
Since \eqref{eq:alb} is direct for $\a = 1$, it suffices to consider the case $\a > 1$.

We want to compute the minimum value of the function
$$
f(x,y) = x + ( 2^{1/\a} - 1 )y
$$
with the restrictions
$g(x,y)= x^{\a} +y^{\a}=1$, $x \ge y \ge 0$.
If $(x,y)$ is a critical point, then there exists $\l \in \RR$ such that
$$
\begin{aligned}
1
& = \l\, \a\, x^{\a-1},
\\
2^{1/\a} - 1
& = \l\, \a\, y^{\a-1},
\end{aligned}
$$
and so, $(y/x)^{\a-1}=2^{1/\a} - 1$ and $y=(2^{1/\a} - 1)^{1/(\a-1)}x$;
this fact and the equality $x^{\a} +y^{\a}=1$ imply
$$
\begin{aligned}
& \big( 1 + (2^{1/\a} - 1)^{\a/(\a-1)} \big) \, x^\a=1 ,
\\
& x = \big( 1 + (2^{1/\a} - 1)^{\a/(\a-1)} \big)^{-1/\a} ,
\\
& y = (2^{1/\a} - 1)^{1/(\a-1)} \big( 1 + (2^{1/\a} - 1)^{\a/(\a-1)} \big)^{-1/\a} ,
\\
f(x,y)
& = \big( 1 + (2^{1/\a} - 1)^{\a/(\a-1)} \big)^{-1/\a}
\\
& \quad + (2^{1/\a} - 1)(2^{1/\a} - 1)^{1/(\a-1)}\big( 1 + (2^{1/\a} - 1)^{\a/(\a-1)} \big)^{-1/\a}
\\
& = \big( 1 + (2^{1/\a} - 1)^{\a/(\a-1)} \big)^{-1/\a}
\\
& \quad + (2^{1/\a} - 1)^{\a/(\a-1)}\big( 1 + (2^{1/\a} - 1)^{\a/(\a-1)} \big)^{-1/\a}
\\
& = \big( 1 + (2^{1/\a} - 1)^{\a/(\a-1)} \big)^{(\a-1)/\a}
> 1.
\end{aligned}
$$

If $y = 0$, then $x = 1$ and $f(x,y) = 1$.

If $y = x$, then $x = 2^{-1/\a} = y$ and
$$
f(x,y)
= 2^{-1/\a} + (2^{1/\a} - 1) 2^{-1/\a}
= 1.
$$

Hence, $f(x,y) \ge 1$ and the bound is tight if and only if $y=0$ or $y=x$.
By homogeneity, we have $f(x,y) \ge 1$ for every $x \ge y \ge 0$ and the bound is tight if and only if $y=0$ or $y=x$.
This finishes the proof of \eqref{eq:alb}.

\smallskip

Consequently,
$$
( d_u^\a + d_v^\a )^{1/\a}
\le d_u + ( 2^{1/\a} - 1 )d_v
= d_u+d_v - ( 2 - 2^{1/\a} ) d_v
$$
for each $\a \ge 1$ and $d_u \ge d_v$.
Thus,
$$
( d_u^\a + d_v^\a )^{1/\a} \le d_u+d_v - ( 2 - 2^{1/\a} ) \d
$$
for each $\a \ge 1$ and $uv \in E(G)$,
and the equality holds for some $\a>1$ if and only if $d_u=d_v= \d$.
Therefore,
$$
SO_\a(G) \le M_1(G) - ( 2 - 2^{1/\a} ) \d ,
$$
and the equality holds for some $\a>1$ if and only if $d_u=d_v= \d$ for every $uv \in E(G)$, i.e., $G$ is regular.
\end{proof}

\begin{corollary} \label{c:alb}
Let $G$ be any graph with maximum degree $\D$ and minimum degree $\d$.
Then
$$
\begin{aligned}
2^{-1/2} M_1(G)
\le SO(G)
& \le M_1(G) - \big( 2 - \sqrt{2} \, \big) \d ,
\end{aligned}
$$
and the equality holds in each bound if and only if $G$ is regular.
\end{corollary}

The upper bound in Corollary \ref{c:alb} appears in \cite[Theorem 2.7]{Das-21}.
Hence, Theorem \ref{t:alb} generalizes \cite[Theorem 2.7]{Das-21}.

\medskip

A family of topological indices, named \emph{Adriatic indices}, was put
forward in \cite{V1,V2}. Twenty of them were selected as significant predictors in Mathematical Chemistry.
One of them, the \emph{inverse sum indeg} index, $ISI$, was singled out
in \cite{V2} as a significant predictor of total surface area of octane isomers.
This index is defined as
$$
ISI(G) = \sum_{uv\in E(G)} \frac{d_u\,d_v}{d_u + d_v}
= \sum_{uv\in E(G)} \frac{1}{\frac{1}{d_u} + \frac{1}{d_v}}\,.
$$
In the last years there is an increasing interest in the mathematical
properties of this index.
We finish this section with two inequalities relating the Sombor, the first Zagreb and the inverse sum indeg indices.

\begin{theorem} \label{t:isi}
Let $G$ be any graph with maximum degree $\D$ and minimum degree $\d$, and $\a \ge 1$.
Then
$$
\sqrt{2}\, \big( M_1(G) - 2 ISI(G) \big)
\ge SO(G)
> M_1(G) - 2 ISI(G)
$$
and the upper bound is tight if and only if $G$ has regular connected components.
\end{theorem}

\begin{proof}
It is well-known that for $x,y>0,$ we have
$$
\begin{aligned}
x^2 + y^2
& < (x + y)^2
\le 2(x^2 + y^2 ),
\\
\sqrt{x^2 + y^2}
& < x + y
\le \sqrt{2} \sqrt{x^2 + y^2 } \,,
\end{aligned}
$$
and the equality
$$
\sqrt{d_u^2 + d_v^2}\sqrt{d_u^2 + d_v^2} + 2 d_u d_v = (d_u + d_v)^2
$$
give
$$
\begin{aligned}
(d_u + d_v) \sqrt{d_u^2 + d_v^2} + 2 d_u d_v
& > (d_u + d_v)^2,
\\
\sqrt{d_u^2 + d_v^2} + \frac{2 d_u d_v}{d_u + d_v}
& > d_u + d_v ,
\\
SO(G) + 2 ISI(G)
& > M_1(G) .
\end{aligned}
$$
In a similar way, we obtain
$$
\begin{aligned}
\frac1{\sqrt{2}} \, (d_u + d_v) \sqrt{d_u^2 + d_v^2} + 2 d_u d_v
& \le (d_u + d_v)^2,
\\
\sqrt{d_u^2 + d_v^2} + \sqrt{2} \, \frac{2 d_u d_v}{d_u + d_v}
& \le \sqrt{2} \, (d_u + d_v) ,
\\
SO(G) + 2\sqrt{2}\, ISI(G)
& \le \sqrt{2} \, M_1(G) .
\end{aligned}
$$

The equality in this last inequality is tight if and only if
$2(d_u^2 + d_v^2 ) = (d_u + d_v)^2$
for each edge $uv$, i.e.,
$d_u = d_v$
for every $uv \in E(G)$, and this happens if and only if $G$ has regular connected components.
\end{proof}

\section{Optimization problems}

We start this section with a technical result.

\begin{proposition} \label{p:edge}
Let $G$ be any graph, $u,v \in V(G)$ with $uv \notin E(G)$, and $\a,\b \in \RR \setminus \{0\}$ with $\a\b > 0$.
Then $KA_{\a,\b}(G \cup \{uv\}) > KA_{\a,\b}(G)$.
If $\a>0$, then ${}_{red}KA_{\a,\b}(G \cup \{uv\}) > {}_{red}KA_{\a,\b}(G)$.
Furthermore, if $\a<0$ and $G$ does not have pendant vertices, then ${}_{red}KA_{\a,\b}(G \cup \{uv\}) > {}_{red}KA_{\a,\b}(G)$.
\end{proposition}

\begin{proof}
Let $\{w_1,\dots,w_{d_u}\}$ and $\{w^1,\dots,w^{d_v}\}$ be the sets of neighbors of $u$ and $v$ in $G$, respectively.
Since $\a\b > 0$, the function
$$
U(x,y)
= \big(x^\a + y^\a \big)^\b
$$
is strictly increasing in each variable if $x,y>0$.
Hence,
$$
\begin{aligned}
KA_{\a,\b}(G \cup \{uv\})- KA_{\a,\b}(G)
& = \big( (d_u+1)^\a + (d_v+1)^\a \big)^\b
+
\\
& \quad + \sum_{j=1}^{d_u} \Big( \, \big( (d_u+1)^\a + d_{w_j}^\a \big)^\b - \big( d_u^\a + d_{w_j}^\a \big)^\b \, \Big)
\\
& \quad + \sum_{k=1}^{d_v} \Big( \, \big( (d_v+1)^\a + d_{w^k}^\a \big)^\b - \big( d_v^\a + d_{w^k}^\a \big)^\b \, \Big)
\\
& > \big( (d_u+1)^\a + (d_v+1)^\a \big)^\b
> 0.
\end{aligned}
$$
The same argument gives the results for the ${}_{red}KA_{\a,\b}$ index.
\end{proof}

\smallskip

Given an integer number $n\ge 2$, let $\Gamma(n)$ (respectively, $\Gamma_{\!\!c}(n)$) be the set of graphs
(respectively, connected graphs) with $n$ vertices.

\smallskip

We study in this section the extremal graphs for the $KA_{\a,\b}$ index on $\Gamma_{\!\!c}(n)$ and $\Gamma(n)$.

\begin{theorem} \label{t:op2}
Consider $\a,\b \in \RR \setminus \{0\}$ with $\a\b > 0$, and an integer $n\ge 2$.

$(1)$ The complete graph $K_n$ is the unique graph that maximizes $KA_{\a,\b}$ on $\Gamma_{\!\!c}(n)$ or $\Gamma(n)$.

$(2)$ Any graph that minimizes $KA_{\a,\b}$ on $\Gamma_{\!\!c}(n)$ is a tree.

$(3)$ If $n$ is even, then the union of $n/2$ paths $P_2$ is the unique graph that minimizes $KA_{\a,\b}$ on $\Gamma(n)$.
If $n$ is odd, then the union of $(n-3)/2$ paths $P_2$ with a path $P_3$ is
the unique graph that minimizes $KA_{\a,\b}$ on $\Gamma(n)$.

$(4)$ Furthermore, if $\a,\b > 0$, then the three previous statements hold if we replace $KA_{\a,\b}$ with $_{red}KA_{\a,\b}$.
\end{theorem}

\begin{proof}
Let us denote by $m$ the cardinality of the set of edges of a graph, and by $\d$ its minimum degree.

Assume that $n$ is even.
It is well known that the sum of the degrees of a graph is equal to twice the number of edges of the graph (handshaking lemma).
Thus, $2m \ge n\d \ge n$.
Since $\a\b > 0$, the function
$$
U(x,y)
= \big(x^\a + y^\a \big)^\b
$$
is strictly increasing in each variable if $x,y>0$.
Hence, for any graph $G \in \Gamma(n)$, we have
$$
\begin{aligned}
KA_{\a,\b}(G)
& = \sum_{uv \in E(G)} \big(d_u^\a + d_{v}^\a \big)^\b
\ge \sum_{uv \in E(G)} \big(1^\a + 1^\a \big)^\b
\\
& = 2^\b m
\ge 2^\b \, \frac{n}{2}
= 2^{\b-1} n,
\end{aligned}
$$
and the equality is tight in the inequality if and only if $d_u = 1$ for all $u \in V(G)$, i.e.,
$G$ is the union of $n/2$ path graphs $P_2$.

Finally, assume that $n$ is odd.
Fix a graph $G \in \Gamma(n)$.
If $d_u=1$ for every $u \in V(G)$, then handshaking lemma gives $2m = n$, a contradiction (recall that $n$ is odd).
Therefore, there exists a vertex $w$ with $d_w \ge 2$.
By handshaking lemma we have $2m \ge (n-1) \d+2 \ge n +1$.
Recall that the set of neighbors of the vertex $w$ is denoted by $N(w)$.
Since $U(x,y)$ is a strictly increasing function in each variable, we obtain
$$
\begin{aligned}
KA_{\a,\b}(G)
& = \sum_{u \in N(w)} \big(d_u^\a + d_{w}^\a \big)^\b
+ \sum_{uv \in E(G), u,v\neq w} \big(d_u^\a + d_{v}^\a \big)^\b
\\
& \ge \sum_{u \in N(w)} \big(1^\a + 2^\a \big)^\b
+ \sum_{uv \in E(G), u,v\neq w} \big(1^\a + 1^\a \big)^\b
\\
& \ge 2 \big(1 + 2^\a \big)^\b
+ 2^\b (m-2)
\\
& \ge 2 \big(1 + 2^\a \big)^\b
+ 2^\b \Big( \frac{n+1}{2}-2 \Big)
\\
& = 2 \big(1 + 2^\a \big)^\b
+ 2^\b \frac{n-3}{2} \,,
\end{aligned}
$$
and the bound is tight if and only if $d_u = 1$ for all $u \in V(G)\setminus \{w\}$, and $d_w = 2$.
Hence, $G$ is the union of $(n-3)/2$ path graphs $P_2$ and a path graph $P_3$.

Items $(1)$ and $(2)$ follow directly from Proposition \ref{p:edge}.

If $\a,\b > 0$, then the same argument gives the results for the $_{red}KA_{\a,\b}$ index.
\end{proof}

We deal now with the optimization problem for $_{red}KA_{\a,\b}$
when $\a,\b < 0$.

Given an integer number $n\ge 3$, we denote by $\Gamma^{wp}(n)$ (respectively, $\Gamma_{\!\!c}^{wp}(n)$) the set of graphs
(respectively, connected graphs) with $n$ vertices and without pendant vertices.

\begin{theorem} \label{t:op2}
Consider $\a,\b < 0$, and an integer $n\ge 3$.

$(1)$ The complete graph $K_n$ is the unique graph that maximizes $_{red}KA_{\a,\b}$ on $\Gamma_{\!\!c}^{wp}(n)$ or $\Gamma^{wp}(n)$.

$(2)$ The cycle graph $C_n$ is the unique graph that minimizes $_{red}KA_{\a,\b}$ on $\Gamma_{\!\!c}^{wp}(n)$.

$(3)$ The union of cycle graphs are the only graphs that minimize $_{red}KA_{\a,\b}$ on $\Gamma^{wp}(n)$.
\end{theorem}

\begin{proof}
Since a graph without pendant vertices satisfies $\d \ge 2$, handshaking lemma gives $2m \ge n\d \ge 2n$.
Since $\a,\b < 0$, the function
$$
U(x,y)
= \big(x^\a + y^\a \big)^\b
$$
is strictly increasing in each variable if $x,y>0$.
Hence, for any graph $G \in \Gamma^{wp}(n)$, we have
$$
\begin{aligned}
KA_{\a,\b}(G)
& = \sum_{uv \in E(G)} \big(d_u^\a + d_{v}^\a \big)^\b
\ge \sum_{uv \in E(G)} \big(2^\a + 2^\a \big)^\b
\\
& = 2^{(\a+1)\b} m
\ge 2^{(\a+1)\b} n,
\end{aligned}
$$
and the inequality is tight if and only if $d_u = 2$ for all $u \in V(G)$, i.e.,
the graph $G$ is the union of cycle graphs.
If $G$ is connected, then it is the cycle graph $C_n$.

Item $(1)$ follows from Proposition \ref{p:edge}.
\end{proof}


\section*{}


\begin{thebibliography}{00}



\bibitem{Albertson}  M. O. Albertson, The irregularity of a graph,
        {\it Ars Comb.} {\bf 46} (1997) 219--225.
        
\bibitem{AGMM} A. Ali, I. Gutman, E. Milovanovi\'c, I. Milovanovi\'c,
Sum of powers of the degrees of graphs: Extremal results and bounds.
{\it MATCH Commun. Math. Comput. Chem.} {\bf 80} (2018) 5--84.

\bibitem{Bor-17}
B. Borovi\'canin, K. C. Das, B. Furtula, I. Gutman, Bounds for Zagreb indices,
{\it MATCH Commun. Math. Comput. Chem.} {\bf 78} (2017) 17--100.

\bibitem{BMRS}
P. Bosch, E. Molina, J. M. Rodr{\'\i}guez, J. M. Sigarreta,
Inequalities on the Generalized ABC Index,
{\it Mathematics} {\bf 9(10)} (2021) 1151.

\bibitem{BSG} G. Britto Antony Xavier, E. Suresh, I. Gutman, Counting relations for general Zagreb indices,
\emph{Kragujevac J. Math.\/} 38 (2014) 95--103.


\bibitem{Cru-21} R. Cruz, I. Gutman, J. Rada, Sombor index of chemical
graphs, {\it Appl. Math. Comput.} {\bf 399} (2021) 126018.

\bibitem{Cru-21b} R. Cruz, J. Rada, Extremal values of the Sombor index in
unicyclic and bicyclic graphs, to appear in {\it J. Math. Chem.}

\bibitem{Das-21} K. C. Das, A.S. \c{C}evik, I.N. Cangul, Y. Shang, On Sombor index, {\it Symmetry} {\bf 13} (2021) 140.

\bibitem{Gut-14} I. Gutman, Degree-based topological indices, {\it Croat. Chem. Acta} {\bf 86(4)} (2014) 351--361.

\bibitem{Gut-21} I. Gutman, Geometric approach to degree-based topological
indices: Sombor indices, {\it MATCH Commun. Math. Comput. Chem.} {\bf 86(1)} (2021) 11--16.

\bibitem{Gut-21b} I. Gutman, Some basic properties of Sombor indices,
{\it Open J. Discret. Appl. Math.} {\bf 4(1)} (2021) 1--3.

\bibitem{Gut-18} I. Gutman, E. Milovanovi\'{c}, I. Milovanovi\'{c},
Beyond the Zagreb indices, {\it AKCE Int. J. Graphs Comb.} {\bf 17} (2018) 74--85.

\bibitem{Gut-Tri-72} I. Gutman, N. Trinajsti\'c, Graph theory and molecular orbitals. Total $\phi$-electron energy of alternant hydrocarbons,
{\it Chem. Phys. Letters} {\bf 17(4)} (1972) 535--538.

\bibitem{EIG} M. Eliasi, A. Iranmanesh, I. Gutman, Multiplicative versions of first Zagreb index,
{\it MATCH Commun. Math. Comput. Chem.} {\bf 68(1)} (2012) 217--230.

\bibitem{Kober} H. Kober, On the arithmetic and geometric means and on H\"older's inequality, {\it Proc. Amer. Math. Soc.} {\bf 9} (1958) 452--459.
    
\bibitem{LZhao} Li, X., Zhao, H.: Trees with the first smallest and largest generalized topological indices, 
\emph{MATCH Commun. Math. Comput. Chem.\/} {\bf 50} (2004) 57--62.

\bibitem{LZheng} X. Li, J. Zheng, A unified approach to the extremal trees for different indices,
\emph{MATCH Commun. Math. Comput. Chem.\/} {\bf 54} (2005) 195--208.

\bibitem{MRS-S}
A. Mart{\'\i}nez-P\'erez, J. M. Rodr{\'\i}guez, J. M. Sigarreta,
A new approximation to the geometric-arithmetic index,
{\it J. Math. Chem.} {\bf 56} (2018) 1865--1883.

\bibitem{MN} A. Mili\v{c}evi\'c, S. Nikoli\'c, On variable Zagreb indices, \emph{Croat. Chem. Acta\/} 77 (2004) 97--101.

\bibitem{Mil-21} I. Milovanovi\'{c}, E. Milovanovi\'{c}, M. Mateji\'{c},
On some mathematical properties of Sombor indices, {\it Bull. Int. Math. Virtual Inst.} {\bf 11(2)} (2021) 341--353.

\bibitem{NMTJ} S. Nikoli\'c, A. Mili\v{c}evi\'c, N. Trinajsti\'c, A. Juri\'c,
On Use of the Variable Zagreb $^\nu M_2$ Index in QSPR: Boiling Points of Benzenoid Hydrocarbons, {\em Molecules\/} 9 (2004) 1208--1221.

\bibitem{Rada-21}
Rada, J.;  Rodr{\'\i}guez, J. M.; Sigarreta, J. M.
General properties on Sombor indices.
{\it Discr. Appl. Math.} (in~press).

\bibitem{Red-21} I. Red\v{z}epovi\'{c}, Chemical applicability of Sombor
indices, to appear in {\it J. Serb. Chem. Soc.} (2021).


\bibitem{RDA-21}
T. R\'eti, T. Do\v{s}li\'c, A. Ali, On the Sombor index of graphs,
{\it Contrib. Math.} {\bf 3} (2021) 11--18.

\bibitem{Shang-16} Y. Shang, Groupies in multitype random graphs, {\it Springerplus} {\bf 5(1)} (2016) 989.

\bibitem{Tod-09} R. Todeschini, V. Consonni, Molecular Descriptors for Chemoinformatics, Wiley-VCH, Weinheim, 2009.

\bibitem{TC} R. Todeschini, V. Consonni, New local vertex invariants and molecular descriptors based on functions of the vertex degrees, {\it MATCH Commun. Math. Comput. Chem.} {\bf 64(2)} (2010) 359--372.

\bibitem{ZGA} B. Zhou, I. Gutman, T. Aleksi\'c, A note on Laplacian energy of graphs, {\it MATCH Commun. Math. Comput. Chem.} {\bf 60} (2008) 441--446.

\bibitem{G21a}
I. Gutman,
Geometric approach to degree-based topological indices: Sombor indices.
MATCH Commun. Math. Comput. Chem. {\bf 86}, 11--16 (2021).

\bibitem{KG21}
V. R. Kulli and I. Gutman,
Computation of Sombor indices of certain networks,
SSRG Int. J. Appl. Chem. {\bf 8}, 1--5 (2021).

\bibitem{LZKM21}
Z. Lina, T. Zhoub, V. R. Kullic, and L. Miao,
On the first Banhatti-Sombor index,
preprint arXiv:2104.03615.

\bibitem{RDA21}
T. Reti, T. Doslic, and A Ali,
On the Sombor index of graphs,
Contrib. Math. {\bf 3}, 11--18 (2021).

\bibitem{K19}
V. R. Kulli,
The $(a,b)-KA$ indices of polycyclic aromatic hydrocarbons and benzenoid systems,
International Journal of Mathematics Trends and Technology {\bf 65}, 115--120 (2019).

\bibitem{V1} Vuki\v{c}evi\'c, D.; Ga\v{s}perov, M. Bond additive modeling 1. Adriatic indices.
{\em Croat. Chem. Acta} {\bf 2010}, {\em 83}, 243--260.

\bibitem{V2} Vuki\v{c}evi\'c, D. Bond additive modeling 2. Mathematical properties of max-min rodeg index. 
{\em Croat. Chem. Acta} {\bf 2010}, {\em 83}, 261--273.

\bibitem{ZT10}
B. Zhou and N. Trinajsti\'c,
On general sum-connectivity index,
J. Math. Chem. {\bf 47}, 210--218 (2010).

\bibitem{K21}
V. R. Kulli,
$\delta$-Sombor index and its exponential for certain nanotubes,
Annals of Pure and Applied Mathematics, in press (2021).

\end{thebibliography}
\end{document}